\documentclass[preprint,10pt]{elsarticle}

\usepackage{amssymb}
\usepackage{amsmath}
\usepackage{amsfonts}
\usepackage{amssymb}
\usepackage{indentfirst,latexsym,bm}
\usepackage{amsthm}
\usepackage[all]{xy}
\usepackage{tikz}
\usepackage{pgflibraryarrows}
\usepackage{pgflibrarysnakes}

\newtheorem{theorem}{Theorem}[section]
\newtheorem{lemma}[theorem]{Lemma}
\newtheorem{corollary}[theorem]{Corollary}

\theoremstyle{definition}

\theoremstyle{definition}
\newtheorem{ex}{Example}[section]

\theoremstyle{remark}
\newtheorem{rem}{Remark}[section]

\numberwithin{equation}{section}

\journal{XXX}

\begin{document}

\begin{frontmatter}
\allowdisplaybreaks


\title{Perturbation estimation for the parallel sum of Hermitian positive semi-definite matrices}

\author[rvt]{Wei Luo}
\ead{luoweipig1@163.com}
\author[rvt]{Chuanning Song}
\ead{songning@shnu.edu.cn}
\author[rvt]{Qingxiang Xu\corref{cor1}\fnref{fn1}}
\ead{qxxu@shnu.edu.cn,qingxiang\_xu@126.com}
\cortext[cor1]{Corresponding author}
\fntext[fn1]{Supported by the
National Natural Science Foundation of China (11671261).}
\address[rvt]{Department of Mathematics, Shanghai Normal University, Shanghai 200234, PR China}

\begin{abstract}Let $\mathbb{C}^{n\times n}$ be the set of all $n \times n$ complex matrices. For any Hermitian positive semi-definite matrices $A$ and $B$ in $\mathbb{C}^{n\times n}$, their new common upper bound less than $A+B-A:B$  is constructed, where $(A+B)^\dag$ denotes the Moore-Penrose inverse of $A+B$, and $A:B=A(A+B)^\dag B$ is the parallel sum of $A$ and $B$. A
factorization formula for $(A+X):(B+Y)-A:B-X:Y$ is derived, where $X,Y\in\mathbb{C}^{n\times n}$ are any Hermitian positive semi-definite perturbations of $A$ and $B$, respectively.  Based on the derived factorization formula  and the constructed common upper bound of $X$ and $Y$, some new and sharp norm upper bounds of $(A+X):(B+Y)-A:B$ are provided. Numerical examples are also provided to illustrate the sharpness of the obtained norm upper bounds.
\end{abstract}


\end{frontmatter}



\section{Introduction and preliminaries}
Throughout this paper,  $\mathbb{C}^{m\times n}$ is the set of all $m \times n$ complex matrices. For any $A\in \mathbb{C}^{m\times n}$, let ${\cal R}(A)$, $A^*$ and $\Vert A \Vert$ denote the range, the conjugate transpose and the 2-norm of $A$, respectively. Let $B,C\in\mathbb{C}^{n\times n}$ be Hermitian. The notation $B\ge C$ is used to indicate that $B-C$ is positive semi-definite.

Recall that the Moore-Penrose inverse of $A\in \mathbb{C}^{m\times n}$ is the unique element $A^\dag\in \mathbb{C}^{n\times m}$ which satisfies
$$AA^\dag A=A,\ A^\dag A A^\dag=A^\dag,\ (AA^\dag)^*=AA^\dag\ \mbox{and}\ (A^\dag A)^*=A^\dag A.$$
It is known (see e.g.\cite{wang-wei-qiao}) that $\mathcal{R}(A^\dag)=\mathcal{R}(A^*)$, $(A^*)^\dag=(A^\dag)^*$ and
$(AA^*)^\dag=(A^*)^\dag A^\dag$.
So if $A$ is Hermitian, then $(A^\dag)^*=A^\dag$ and $AA^\dag=A^\dag A$.

One application of the Moore-Penrose inverse is the study of the parallel sum introduced by Anderson and Duffin in \cite{Anderson-Duffin} for Hermitian positive semi-definite matrices. Let $A,B\in\mathbb{C}^{n\times n}$ be Hermitian positive semi-definite. The parallel sum of $A$ and $B$ is defined by \begin{equation}\label{equ:notation of parallel sum}A:B=A(A+B)^\dag B,\end{equation}
which is so named because of its origin in and application to the electrical network theory that
$$\big(r_1^{-1}+r_2^{-1}\big)^{-1}=r_1(r_1+r_2)^{-1}r_2$$ is the resistance arising  from resistors $r_1$ and $r_2$ in parallel.
It is proved in \cite[Lemmas~2 and 4]{Anderson-Duffin} that $A:B\ge 0$  and a norm upper bound of $A:B$ is given in
\cite[Theorem~25]{Anderson-Duffin} as
\begin{equation}\label{equ:sharp upper bound for parallel sum wrt positive operators}\Vert A:B\Vert\le \frac{\Vert A\Vert\cdot\Vert B\Vert}{\Vert A\Vert+\Vert B\Vert}.\end{equation}
The perturbation estimation for the parallel sum is also considered in \cite{Anderson-Duffin}. More precisely, let $A,B, X, Y\in\mathbb{C}^{n\times n}$ be all Hermitian positive semi-definite and let $E$ be the error induced by the perturbation of $A:B$ as
 \begin{equation}\label{eqn:defn of E}E=(A+ X):(B+ Y)-A:B.\end{equation} It can be deduced from \cite[Corollary~21]{Anderson-Duffin} that $E\ge 0$  and a norm upper bound of $E$ is established in \cite[Theorem~31]{Anderson-Duffin} as
\begin{equation}\label{equ:Anderson-Duffin's upper bound}\Vert E\Vert \le \lambda_{A,B}\,\Vert X+Y\Vert,\end{equation}
where
\begin{equation}\label{equ:defn of lambda A B}\lambda_{A,B}=2\Vert (A+B)^\dag A\Vert^2+2\Vert (A+B)^\dag B\Vert^2+\frac12.\end{equation}

Ever since the publication of \cite{Anderson-Duffin}, the parallel sum has been studied in the more general settings of non-square matrices under certain conditions of range inclusions \cite{Mitra-Odell}, of positive operators $A$ and $B$ on a Hilbert space such that the range of $A+B$ is closed \cite{Anderson-Schreiber} and furthermore,  without any assumptions on the range of $A+B$ \cite{Fillmore-Williams,Morley}. As the generalizations of the parallel sum,  shorted operators and the weakly parallel sum are also studied in \cite{Anderson,Anderson-Trapp,Butler-Morley,Mitra,Mitra-Puntanen} and \cite{Antezana-Corach-Stojanoff,Djikic}, respectively. For many different equivalent definitions and the properties of the parallel sum, see a recent review paper \cite{Berkics} and the references therein.

Although much progress has been made in the study of the parallel sum and its various generalizations, very little has been done on the improvement of norm upper bound \eqref{equ:Anderson-Duffin's upper bound}, which is the concern of this paper. Let $X,Y\in\mathbb{C}^{n\times n}$ be Hermitian positive semi-definite. Checking the proof of \cite[Theorem~31]{Anderson-Duffin} carefully, we find that
the norm $\Vert X+Y\Vert$ appearing in \eqref{equ:Anderson-Duffin's upper bound} can in fact be replaced by any $\Vert Z\Vert$, where $Z$ is any common upper bound of $X$ and $Y$. The less is  $Z$, the sharper is the resulting upper bound \eqref{equ:Anderson-Duffin's upper bound}. This leads us to investigate small common upper bounds of $X$ and $Y$. One choice less than $X+Y$ is the matrix $C_{X,Y}$ defined by
\begin{equation}\label{equ:defn of C X Y}C_{X,Y}=X+Y-X:Y,\end{equation}
which is a common upper bound of $X$ and $Y$ since by \eqref{equ:two sides A or two sides B}, $C_{X,Y}-X=Y-X:Y=Y(X+Y)^\dag Y\ge 0$ and $C_{X,Y}-Y\ge 0$ in a similar way.
It is of independent interest to find out a common upper bound of $X$ and $Y$, which is even less than $C_{X,Y}$. By using certain $C^*$-algebraic technique, we have managed to figure out such a common upper bound $X\vee Y$; see Theorem~\ref{thm:Almost least common upper bound} for the details.

Another way to improve the upper bound \eqref{equ:Anderson-Duffin's upper bound} is the reduction of the coefficient $\lambda_{A,B}$ given by  \eqref{equ:defn of lambda A B}, where $A,B,X,Y\in \mathbb{C}^{n\times n}$ are all Hermitian positive semi-definite.   Let $T$ and $H$ be defined by
\begin{equation}\label{equ:defn of H}T=(A+B):(X+Y),\ H=(A+X):(B+Y)-A:B-X:Y.\end{equation}
Along the line of checking $\langle Hx,x\rangle\ge 0\ \mbox{for any $x\in\mathbb{C}^n$}$, it is proved in \cite[Lemmas~18 and 20]{Anderson-Duffin} that the matrix $H$ defined by \eqref{equ:defn of H} is also Hermitian positive semi-definite. An interpretation of such a result with electronic circuits is as follows:
\\
{\center
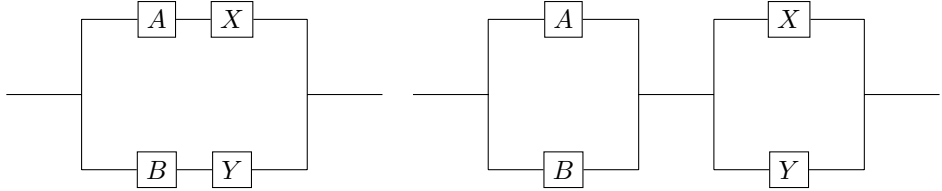
\begin{tikzpicture}

\coordinate (P0) at (0,0);
\coordinate (P1) at (1,0);
\coordinate (P2) at (1,1);
\coordinate (P3) at (1,-1);
\coordinate (P4) at (4,1);
\coordinate (P5) at (4,-1);
\coordinate (P6) at (4,0);
\coordinate (P7) at (5,0);

\tikzstyle{every node}=[draw, shape=rectangle];
\node (A) at (2,1) {$A$};
\node (B) at (2,-1) {$B$};
\node (X) at (3,1) {$X$};
\node (Y) at (3,-1) {$Y$};

\draw (P0)--(P1) (P1)--(P2) (P1)--(P3)
(P2)--(A) (A)--(X) (X)--(P4) (P4)--(P6) (P6)--(P7) (P3)--(B)
(B)--(Y) (Y)--(P5) (P5)--(P6)   ;
\end{tikzpicture}
\,\,\,\,
\begin{tikzpicture}
\coordinate (P0) at (0,0);
\coordinate (P1) at (1,0);
\coordinate (P2) at (1,1);
\coordinate (P3) at (1,-1);
\coordinate (P4) at (3,1);
\coordinate (P5) at (3,-1);
\coordinate (P6) at (3,0);
\coordinate (P7) at (4,0);
\coordinate (P8) at (4,1);
\coordinate (P9) at (4,-1);
\coordinate (P10) at (6,1);
\coordinate (P11) at (6,-1);
\coordinate (P12) at (6,0);
\coordinate (P13) at (7,0);

\tikzstyle{every node}=[draw, shape=rectangle];
\node (A) at (2,1) {$A$};
\node (B) at (2,-1) {$B$};
\node (X) at (5,1) {$X$};
\node (Y) at (5,-1) {$Y$};

\draw (P0)--(P1) (P1)--(P2) (P1)--(P3)
(P2)--(A) (A)--(P4) (P4)--(P6) (P3)--(B) (B)--(P5) (P5)--(P6) (P6)--(P7) (P7)--(P8) (P7)--(P9) (P8)--(X) (X)--(P10) (P10)--(P12)
(P9)--(Y) (Y)--(P11) (P11)--(P12) (P12)--(P13)  ;
\end{tikzpicture}
}
\\
The key point of this paper is,
a factorization formula for $H$ can be derived as \eqref{equ:key equality concerning perturbation of parallel sum}, which leads obviously to
the positivity of $H$ since by Lemma~\ref{lem:two sides A or two sides B}, the matrix $T$ defined by \eqref{equ:defn of H} is positive.

The parallel sum and its generalizations have proved to be useful operations in a wide variety of
fields, such as electrical networks \cite{Anderson-Duffin,Andersont-Morley-Trapp,Mitra}, statistics \cite{Mitra-Puntanen,Ouellette},
control theory \cite{{Andersont-Morley-Trapp-2},Green-Kamen}, geodetic adjustments \cite{Zhong-Welsch}, image denoising
problems \cite{Bot-Hendrich}, signal recovery \cite{Becker-Combettes}, numerical calculations \cite{Berlinet} and so on.  In view of the observational error or measuring error, it is meaningful to study the perturbation estimation of the parallel sum.

Formula \eqref{equ:key equality concerning perturbation of parallel sum} plays a crucial role in our study of the  perturbation estimation for the parallel sum. It is firstly applied  to study the one-sided perturbation (\ref{equ:defn of G-one side}), and is then applied to deal with the special case of the two-sided perturbation \eqref{eqn:defn of F two dides X}, where a norm upper bound (\ref{eqn:1st sharp of norm F}), as well as its simplified version (\ref{eqn:2nd sharp of norm F}), is obtained. The general two-sided perturbation \eqref{eqn:defn of E} is concerned in Theorem~\ref{thm:norm upper bounds with parameters}, where two norm upper bounds with parameters are derived, and one of which turns out to be the infimum of a function $f(t)$ defined on $(0,+\infty)$ as \eqref{equ:defn of f}. As shown by Example~\ref{ex:parameter to be selected}, this
infimum is easy to handle since the parameter $t$ can be chosen by using certain Matlab commands directly. The sharpness of the newly obtained upper bounds are illustrated by   Remark~\ref{rem:comparison of upper bound wrt Z} and two numerical examples in Section~\ref{sec:Numerical examples}.

The rest of this paper is organized as follows. In Section~\ref{sec:new common upper bound}, a new kind of common upper bound of two Hermitian positive semi-definite matrices   is constructed. In Section~\ref{sec:perturbation analysis}, the perturbation estimation for the parallel sum of Hermitian positive semi-definite matrices  is carried out. In Section~\ref{sec:Numerical examples}, two numerical examples  are  provided.

\section{A new kind of common upper bound of two Hermitian positive semi-definite matrices}\label{sec:new common upper bound}
The purpose of this section is to construct a new kind of common upper bound of two Hermitian positive semi-definite matrices. We begin with two auxiliary lemmas, whose proofs are direct.
\begin{lemma}\label{lem:some trivial propositions of M-P inverse} For any $A\in\mathbb{C}^{m\times n}$, it holds that ${\cal R}(AA^*)={\cal R}(A)$. If in addition $m=n$ and $A\ge 0$, then ${\cal R}(A^\frac12)={\cal R}(A)$, $A^\dag\ge 0$ and $(A^\dag)^\frac12=(A^\frac12)^\dag$.
\end{lemma}

\begin{lemma}\label{lem:large positive moore-penrose invertible implies large range} Let $A,B\in\mathbb{C}^{n\times n}$ be such that $0\le A\le B$. Then ${\cal R}(A)\subseteq {\cal R}(B)$ and $\Vert A\Vert\le \Vert B\Vert$.
 \end{lemma}

Some basic properties of the parallel sum of Hermitian positive semi-definite matrices are derived in \cite{Anderson-Duffin}, part of which are as follows:
\begin{lemma}\label{lem:two sides A or two sides B}{\rm (cf.\,\cite[Lemmas~1--4]{Anderson-Duffin} and Lemma~\ref{lem:large positive moore-penrose invertible implies large range})}\ Let $A,B\in \mathbb{C}^{n\times n}$ be Hermitian positive semi-definite. Then $A:B\ge 0$, $\mathcal{R}(A:B)=\mathcal{R}(A)\cap \mathcal{R}(B)$, and  \begin{equation}\label{equ:two sides A or two sides B}A:B=B:A=A-A(A+B)^\dag A=B-B(A+B)^\dag B.\end{equation}
\end{lemma}

Now, we state the main result of this section as follows:
\begin{theorem}\label{thm:Almost least common upper bound}Let $X,Y\in\mathbb{C}^{n\times n}$ be both Hermitian positive semi-definite. Then $X+Y-4(X:Y)\ge 0$, and a common upper bound of $X$ and $Y$  can be given by
\begin{equation}\label{equ:Introduced Z}X\vee Y=\frac{X+Y}{2}+(X+Y)^\frac12\, W^\frac12\, (X+Y)^\frac12,
\end{equation}
where
\begin{equation}\label{equ:Introduced W}W=\Big[(X+Y)^\dag\Big]^\frac12\Big[\frac{X+Y}{4}-X:Y\Big]\Big[(X+Y)^\dag\Big]^\frac12.\end{equation}
Furthermore, $X\vee Y\le C_{X,Y}$ and $X\vee Y=C_{X,Y}$ if and only if $X:Y=0$, where
$C_{X,Y}$ is defined by \eqref{equ:defn of C X Y}.
\end{theorem}
\begin{proof}(1) We prove that the matrix $X\vee Y$ defined by (\ref{equ:Introduced Z}) is a common upper bound of $X$ and $Y$.
Let $P=(X+Y)(X+Y)^\dag$. Then $P$ is an orthogonal projection and by Lemma~\ref{lem:some trivial propositions of M-P inverse}, we have
$$P=(X+Y)^\frac12 \Big[(X+Y)^\frac12\Big]^\dag=(X+Y)^\frac12 \Big[(X+Y)^\dag\Big]^\frac12=\Big[(X+Y)^\dag\Big]^\frac12 (X+Y)^\frac12,$$
which is the unit of the $C^*$-subalgebra $\mathfrak{B}$ of ${\cal L}(H)$ defined by
$$\mathfrak{B}=P{\cal L}(H)P=\Big[(X+Y)^\dag\Big]^\frac12 {\cal L}(H) \Big[(X+Y)^\dag\Big]^\frac12,$$
where $H=\mathbb{C}^{n}$ is a Hilbert space endowed with the usual inner product and $\mathcal{L}(H)\cong \mathbb{C}^{n\times n}$ is the set of all (bounded) linear operators on $H$. Let $X_1,Y_1\in \mathfrak{B}$ be Hermitian positive semi-definite defined by
\begin{equation}\label{equ:defn of X1 and Y1}X_1=\Big[(X+Y)^\dag\Big]^\frac12 X\Big[(X+Y)^\dag\Big]^\frac12\ \mbox{and}\ Y_1=\Big[(X+Y)^\dag\Big]^\frac12 Y\Big[(X+Y)^\dag\Big]^\frac12.\end{equation}
Then clearly, $X_1+Y_1=P$ and hence
$$X_1Y_1=X_1(P-X_1)=X_1-X_1^2=(P-X_1)X_1=Y_1X_1.$$
Let $C^*\big(P,X_1)$ be the unital commutative $C^*$-subalgebra of $\mathfrak{B}$ generated by $P$ and $X_1$, and $Sp(X_1)$ be the spectrum of $X_1$. Then by \cite[Section~1.1]{pedersen}, we know that $C^*\big(P,X_1)$ is isomorphic to $C\big(Sp(X_1)\big)$ via
Gelfand transform $\wedge$ such that
\begin{eqnarray*}\widehat{X_1}(t)=t, \,\widehat{Y_1}(t)=1-t\ \mbox{and}\ \widehat{P}(t)=1,\ \mbox{for any $t\in Sp(X_1)\subseteq [0,+\infty)$}.
\end{eqnarray*}

Now, we let $X_1\vee Y_1\in C^*\big(P,X_1)$ be such that
\begin{equation}\label{equ:expression of Gelfand transform}\widehat{X_1\vee Y_1}(t)=\max\Big\{\widehat{X_1}(t),\widehat{Y_1}(t)\Big\}=\frac12+\sqrt{\Big(t-\frac12\Big)^2}=\frac12+\sqrt{t^2-t+\frac14}.
\end{equation}
Then, clearly $X_1\vee Y_1$ is the least common upper bound of $X_1$ and $Y_1$ in $C^*\big(P,X_1)$.
The expression of $\widehat{X_1\vee Y_1}(t)$ given by  \eqref{equ:expression of Gelfand transform} indicates that
\begin{equation}\label{equ:expression of common upper bound X1 and Y1-1}X_1\vee Y_1=\frac{P}{2}+W_1^\frac12,
\end{equation}
where
\begin{equation}\label{equ:defn of W1}W_1=X_1^2-X_1+\frac{P}{4}=\big(X_1-\frac{P}{2}\big)^*\big(X_1-\frac{P}{2}\big)\ge 0.\end{equation}
In view of (\ref{equ:defn of W1}), (\ref{equ:defn of X1 and Y1}), \eqref{equ:two sides A or two sides B} and (\ref{equ:Introduced W}), we have
\begin{equation}\label{equ:W1 is W}W_1=\Big[(X+Y)^\dag\Big]^\frac12\Big[X(X+Y)^\dag X-X+\frac{X+Y}{4}\Big]\Big[(X+Y)^\dag\Big]^\frac12=W.\end{equation}
The expression of $W_1$ above, together with (\ref{equ:defn of W1}), indicates that
$$\frac{X+Y}{4}-X:Y=(X+Y)^\frac12 W_1 (X+Y)^\frac12\ge 0.$$
Put \begin{equation}\label{equ:defn of new Z1}Z=(X+Y)^\frac12 \cdot( X_1\vee Y_1) \cdot (X+Y)^\frac12.\end{equation}
Then since $X_1\vee Y_1\ge X_1$, we know from (\ref{equ:defn of X1 and Y1}) that
$$Z\ge (X+Y)^\frac12 \cdot X_1 \cdot (X+Y)^\frac12=X.$$ Similarly, it holds that $Z\ge Y$.
Moreover, from (\ref{equ:defn of new Z1}), (\ref{equ:expression of common upper bound X1 and Y1-1}), \eqref{equ:W1 is W} and (\ref{equ:Introduced Z}) we know  that $Z=X\vee Y$. This completes the proof that $X\vee Y$ is a common upper bound of $X$ and $Y$.

(2) We prove that $X\vee Y\le C_{X,Y}$ and $X\vee Y=C_{X,Y}$ if and only if $X:Y=0$. Indeed, by \eqref{equ:defn of C X Y}, \eqref{equ:Introduced Z}, (\ref{equ:two sides A or two sides B}), \eqref{equ:defn of X1 and Y1}, \eqref{equ:W1 is W}  and (\ref{equ:defn of W1}) we have
\begin{eqnarray*}&&C_{X,Y}-X\vee Y=\frac{X+Y}{2}-X:Y-(X+Y)^\frac12\, W^\frac12\, (X+Y)^\frac12\\
&&=(X+Y)^\frac12\left[\frac{P}{2}-\Big[(X+Y)^\dag\Big]^\frac12\,(X:Y)\, \Big[(X+Y)^\dag\Big]^\frac12-W^\frac12\right](X+Y)^\frac12\\
&&=(X+Y)^\frac12\left[\frac{P}{2}-X_1+X_1^2-W_1^\frac12\right](X+Y)^\frac12\\
&&=(X+Y)^\frac12\left[\frac{P}{4}+W_1-W_1^\frac12\right](X+Y)^\frac12\\
&&=(X+Y)^\frac12\left(W_1^\frac12-\frac{P}{2}\right)^2(X+Y)^\frac12=TT^*\ge 0,\end{eqnarray*}
where $T=(X+Y)^\frac12\left(W_1^\frac12-\frac{P}{2}\right)$.
Note that $\Big[(X+Y)^\dag\Big]^\frac12 T=W_1^\frac12-\frac{P}{2}$, so the discussion above indicates that
\begin{eqnarray*}&&C_{X,Y}=X\vee Y\Longleftrightarrow T=0\Longleftrightarrow W_1^\frac12=\frac{P}{2}\Longleftrightarrow W_1=\frac{P}{4}\\
&&\Longleftrightarrow X_1^2-X_1=0\ \mbox{by \eqref{equ:defn of W1}}\\
&&\Longleftrightarrow \Big[(X+Y)^\dag\Big]^\frac12 \Big[X(X+Y)^\dag X-X\Big] \Big[(X+Y)^\dag\Big]^\frac12=0\ \mbox{by \eqref{equ:defn of X1 and Y1}}\\
&&\Longleftrightarrow X:Y=0\ \mbox{by \eqref{equ:two sides A or two sides B}}.
\end{eqnarray*}
This completes the proof of all the assertions.
\end{proof}

Before ending this section, we make a few remarks on the common upper bound \eqref{equ:Introduced Z}. Let $X,Y\in\mathbb{C}^{n\times n}$ be both Hermitian positive semi-definite. If $Z$ is any common upper bound of $X$ and $Y$, then from Lemma~\ref{lem:large positive moore-penrose invertible implies large range} we have $\Vert Z\Vert\ge \max\{\Vert X\Vert, \Vert Y\Vert\}$.
It is interesting to find out a common upper bound which gets the equation above. The matrix $X\vee Y$ defined by \eqref{equ:Introduced Z} is such a common upper bound in the following two cases:

\textbf{Case 1:}\ $X$ and $Y$ are commutative. Indeed, if $XY=YX$, then $X, Y, (X+Y)^\frac12, (X+Y)^\dag
$ and $\big((X+Y)^\dag\big)^\frac12$ are commutative each other. It follows from \eqref{equ:Introduced W} that
\begin{equation*}(X+Y)^\frac12 W^\frac12 (X+Y)^\frac12=\left[(X+Y)W(X+Y)\right]^\frac12,
\end{equation*}
where
\begin{eqnarray*}(X+Y)W(X+Y)&=&(X+Y)\left[\frac{X+Y}{4}-X(X+Y)^\dag Y\right]\\
&=&\frac14 (X+Y)^2-X(X+Y)(X+Y)^\dag Y\\
&=&\frac14 (X+Y)^2-XY\\
&=&\frac14 (X-Y)^2.
\end{eqnarray*}
Accordingly, from \eqref{equ:Introduced Z} we have $X\vee Y=\frac12\left[X+Y+|X-Y|\right]$, which means clearly that
\begin{equation}\label{equ:reach the optimal situation}\Vert X\vee Y\Vert=\max\{\Vert X\Vert, \Vert Y\Vert\}\end{equation} by functional calculus in the commutative $C^*$-algebra generated by
$X$ and $Y$.

\textbf{Case 2:}\ One of $X$ and $Y$ is larger than another. We might as well assume that $X\le Y$. Following the notations as in the proof of
Theorem~\ref{thm:Almost least common upper bound}, we have $X_1\le Y_1$ and thus by \eqref{equ:defn of new Z1}  we conclude that
\begin{eqnarray*}X\vee Y=(X+Y)^\frac12 \cdot( X_1\vee Y_1) \cdot (X+Y)^\frac12=(X+Y)^\frac12\cdot Y_1 \cdot (X+Y)^\frac12=Y,\end{eqnarray*}
which leads to \eqref{equ:reach the optimal situation} obviously.
\begin{ex}{\rm Let $0<a<\frac12$, $X=\left(
                        \begin{array}{cc}
                           a & 0 \\
                          0 & 1 \\
                        \end{array}
                      \right)$ and $Y=\left(
                                        \begin{array}{cc}
                                          1 & 1 \\
                                          1 & 3 \\
                                        \end{array}
                                      \right)$. Then $YX\ne XY$, whereas $0\le X\le Y$.
}\end{ex}

Next, we consider the special case where the underlying matrices are orthogonal projections.
Assume that $P,Q\in\mathbb{C}^{n\times n}$ are two orthogonal projections. Let $P_0$ be the orthogonal projection from $\mathbb{C}^n$ onto $\mathcal{R}(P)\cap \mathcal{R}(Q)$. We prove that
\begin{equation}\label{equ:common upper bound wrt projections}P\vee Q=P+Q-P_0.\end{equation}
Indeed, by \cite[Theorem~8]{Anderson-Duffin} we have $P:Q=\frac12 P_0$, which means that $P_0, P+Q, (P+Q)^\frac12, (P+Q)^\dag$ and $\left((P+Q)^\dag\right)^\frac12$ are commutative each other.
It follows that
\begin{equation}\label{equ:half is the same}(P+Q)^\frac12 W^\frac12 (P+Q)^\frac12=\left[(P+Q)W(P+Q)\right]^\frac12,
\end{equation}
where
\begin{equation*}(P+Q)W(P+Q)=(P+Q)\left[\frac{P+Q}{4}-\frac12 P_0\right]=\frac{(P+Q-2P_0)^2}{4}.
\end{equation*}
Note that $P+Q-2P_0=(P-P_0)+(Q-P_0)\ge 0$, so the equation above indicates that
$\left[(P+Q)W(P+Q)\right]^\frac12=\frac{P+Q-2P_0}{2}$. This, together with \eqref{equ:Introduced Z} and \eqref{equ:half is the same}, yields
\eqref{equ:common upper bound wrt projections}.

Based on \eqref{equ:common upper bound wrt projections}, we prove that
\begin{equation*}\label{equ:reach the optimal situation--}\Vert P\vee Q\Vert=\max\{\Vert P\Vert, \Vert Q\Vert\}\Longleftrightarrow PQ=QP.\end{equation*}
In fact, if $PQ=QP$, then $P\vee Q$ given by \eqref{equ:common upper bound wrt projections} is an orthogonal projection and thus \eqref{equ:reach the optimal situation} is satisfied, with $X$ and $Y$ therein be replaced by $P$ and $Q$, respectively.

On the other hand, if $PQ\ne QP$, then the orthogonal projection $P-P_0$ is non-zero and from \eqref{equ:common upper bound wrt projections} we have
\begin{equation}\label{equ:norm larger than 1-1st}\Vert P\vee Q\Vert\ge \Vert (P-P_0)\cdot P\vee Q \cdot (P-P_0)\Vert=\Vert (P-P_0)+(P-P_0)Q(P-P_0)\Vert.\end{equation}
Similarly, $Q-P_0\ne 0$ and
\begin{equation}\label{equ:norm larger than 1-2nd}\Vert P\vee Q\Vert\ge \Vert (Q-P_0)+(Q-P_0)P(Q-P_0)\Vert.\end{equation}
Suppose on the contrary that $\Vert P\vee Q\Vert=1$, then it can be deduced from \eqref{equ:norm larger than 1-1st} and \eqref{equ:norm larger than 1-2nd} that
$$(P-P_0)Q(P-P_0)=0\ \mbox{and}\ (Q-P_0)P(Q-P_0)=0;$$ or equivalently, $Q(P-P_0)=0$ and $P(Q-P_0)=0$, that is, $QP=P_0$ and $PQ=P_0$, which is in  contradiction to the assumption that $PQ\ne QP$.

\section{Perturbation estimation for the parallel sum}\label{sec:perturbation analysis}
In this section, we study the perturbation estimation for the parallel sum of Hermitian positive semi-definite matrices.
\begin{theorem}\label{thm:key equality concerning perturbation of parallel sum}Suppose that $A,B,X,Y\in \mathbb{C}^{n\times n}$ are all Hermitian positive semi-definite. Let $T$ and $H$ be defined by \eqref{equ:defn of H}.
Then
\begin{equation}\label{equ:key equality concerning perturbation of parallel sum}H=\big[(A+B)^\dag B-(X+Y)^\dag Y\big]^*\cdot T\cdot\big[(A+B)^\dag B-(X+Y)^\dag Y\big].
\end{equation}
\end{theorem}
\begin{proof} For simplicity, we put $S=A+B+X+Y$. Then
\begin{equation}\label{equ:sum of I1 through I4}H=(A+X)S^\dag (B+Y)-A(A+B)^\dag B-X(X+Y)^\dag Y=I_1+I_2+I_3+I_4,\end{equation}
where
\begin{eqnarray*}&&I_1=AS^\dag B-A(A+B)^\dag B, I_2=AS^\dag Y,\\
&&I_3=X S^\dag Y-X(X+Y)^\dag Y, I_4=XS^\dag B.
\end{eqnarray*}
Note that $$A+B\le S, \mathcal{R}\big((A+B)^\dag\big)=\mathcal{R}\big((A+B)^*\big)=\mathcal{R}(A+B)$$ and
$\mathcal{R}(S^\dag S)=\mathcal{R}(SS^\dag)=\mathcal{R}(S)$, so by Lemma~\ref{lem:large positive moore-penrose invertible implies large range} we have
$S^\dag S (A+B)^\dag=(A+B)^\dag$. Similarly, it holds that
$A(A+B)^\dag (A+B)=A$, $(A+B)(A+B)^\dag B=B$ and
$$(A+B)(A+B)^\dag T=T,$$
since $\mathcal{R}(T)=\mathcal{R}(A+B)\cap \mathcal{R}(X+Y)$ by Lemma~\ref{lem:two sides A or two sides B}.
Therefore,
\begin{eqnarray}I_1&=&A\big[S^\dag-(A+B)^\dag\big]B=-AS^\dag\big[S-(A+B)\big](A+B)^\dag B\nonumber\\
&=&-A(A+B)^\dag (A+B)\cdot S^\dag (X+Y)(A+B)^\dag B\nonumber\\
&=&-A(A+B)^\dag T(A+B)^\dag B\nonumber\\
&=&-\big[(A+B)-B\big](A+B)^\dag T(A+B)^\dag B\nonumber\\
\label{eqn:detailed expression of I1}&=&-T(A+B)^\dag B+B(A+B)^\dag T(A+B)^\dag B.
\end{eqnarray}
Similarly, we have
\begin{eqnarray}I_2&=&A(A+B)^\dag (A+B)\cdot S^\dag\cdot (X+Y)(X+Y)^\dag Y\nonumber\\
&=&A(A+B)^\dag T (X+Y)^\dag Y\nonumber\\
\label{eqn:detailed expression of I2}&=&T (X+Y)^\dag Y-B(A+B)^\dag T (X+Y)^\dag Y,\\
\label{eqn:detailed expression of I3}I_3&=&-T(X+Y)^\dag Y+Y(X+Y)^\dag T(X+Y)^\dag Y,\\
\label{eqn:detailed expression of I4}I_4&=&T(A+B)^\dag B-Y(X+Y)^\dag T(A+B)^\dag B.
\end{eqnarray}
Eq.\,(\ref{equ:key equality concerning perturbation of parallel sum}) then follows from (\ref{equ:sum of I1 through I4})--(\ref{eqn:detailed expression of I4}).
\end{proof}

 Now we use Eq.\,(\ref{equ:key equality concerning perturbation of parallel sum}) to study the perturbation estimation for the parallel sum.  First, we consider the one-sided perturbation as follows:

\begin{corollary}\label{cor:one side perturbation} Suppose that $A,B,X\in \mathbb{C}^{n\times n}$ are  all Hermitian positive semi-definite. Let \begin{equation}\label{equ:defn of G-one side} G=(A+X):B-A:B.\end{equation} Then
\begin{equation}\label{equ:norm upper bound for one side perturbation}\Vert G\Vert\le \frac{\Vert (A+B)^\dag B\Vert^2\cdot\Vert A+B\Vert\cdot \Vert X\Vert}{\Vert A+B\Vert+\Vert X\Vert}.\end{equation}
\end{corollary}
\begin{proof} If we put $Y=0$ in \eqref{equ:defn of H}, then a formula for $G$ can be derived immediately from (\ref{equ:key equality concerning perturbation of parallel sum}) as
\begin{equation}\label{equ:one side perturbation}G=\big[(A+B)^\dag B]^*\cdot \big[(A+B):X\big]\cdot \big[(A+B)^\dag B],\end{equation}
which leads obviously to the inequality (\ref{equ:norm upper bound for one side perturbation}) by using norm estimation (\ref{equ:sharp upper bound for parallel sum wrt positive operators}).
\end{proof}

A direct application of the preceding corollary is as follows:
\begin{corollary}{\rm \cite[Theorem~28]{Anderson-Duffin}}\  Suppose that $A,B,X\in \mathbb{C}^{n\times n}$ are  all Hermitian positive semi-definite. Let $G$ be defined by \eqref{equ:defn of G-one side}. Then
\begin{equation}\label{equ:norm upper bound for one side perturbation---}\Vert G\Vert\le \Vert (A+B)^\dag B\Vert^2\cdot \Vert X\Vert.\end{equation}
\end{corollary}

Next, we consider the special case of the two-sided perturbation as follows:

\begin{theorem}\label{thm:two-sided perturbation} Suppose that $A,B,Z\in \mathbb{C}^{n\times n}$ are  all Hermitian positive semi-definite. Let
$\alpha>0,\beta>0$ and
\begin{equation}\label{eqn:defn of F two dides X}F_{\alpha,\beta}=(A+\alpha Z):(B+\beta Z)-A:B.\end{equation}
Then
\begin{eqnarray}\label{eqn:1st sharp of norm F}\Vert F_{\alpha,\beta}\Vert&\le& \frac{1}{\alpha+\beta}\left[\frac{\Vert (A+B)^\dag (\beta A-\alpha B)\Vert^2\cdot \Vert A+B\Vert}{\Vert A+B\Vert+(\alpha+\beta)\Vert Z\Vert}+\alpha\beta\right]\Vert Z\Vert\\
       \label{eqn:2nd sharp of norm F}&\le& \frac{1}{\alpha+\beta}\Big[\Vert (A+B)^\dag (\beta A-\alpha B)\Vert^2+\alpha\beta\Big]\Vert Z\Vert.
\end{eqnarray}
\end{theorem}
\begin{proof}For simplicity, we put
\begin{equation}\label{equ:defn of T S anpha and beta}T_{\alpha,\beta}=(A+B): \big((\alpha+\beta)Z\big)\ \mbox{an}\ S_{\alpha,\beta}=(A+B)^\dag B-\big((\alpha+\beta)Z\big)^\dag (\beta Z).\end{equation}
Since $\mathcal{R}(T_{\alpha,\beta})=\mathcal{R}(A+B)\cap \mathcal{R}(Z)$ by Lemma~\ref{lem:two sides A or two sides B}, we have
\begin{equation*}\label{equ:replacement technique}Z^\dag Z T_{\alpha,\beta}=T_{\alpha,\beta}=(A+B)^\dag (A+B) T_{\alpha,\beta}.\end{equation*}
The equations above, together with \eqref{equ:defn of T S anpha and beta}, yield
\begin{equation}\label{equ:expression of S T alpha and beta}S_{\alpha,\beta}^* T_{\alpha,\beta}=\frac{\alpha B-\beta A}{\alpha+\beta}(A+B)^\dag T_{\alpha,\beta}, T_{\alpha,\beta} S_{\alpha,\beta}=T_{\alpha,\beta}(A+B)^\dag\frac{\alpha B-\beta A}{\alpha+\beta}.
\end{equation}
Note that $(\alpha Z):(\beta Z)=\frac{\alpha\beta}{\alpha+\beta}Z$,  so by \eqref{eqn:defn of F two dides X}, \eqref{equ:defn of H},
\eqref{equ:key equality concerning perturbation of parallel sum}, \eqref{equ:defn of T S anpha and beta}, \eqref{equ:expression of S T alpha and beta}
and \eqref{equ:sharp upper bound for parallel sum wrt positive operators},  we have
\begin{eqnarray*}\Vert F_{\alpha,\beta}\Vert&=&\Vert S_{\alpha,\beta}^* T_{\alpha,\beta} S_{\alpha,\beta}+(\alpha Z):(\beta Z)\Vert\\
&=& \left\Vert \frac{\beta A-\alpha B}{\alpha+\beta}\cdot (A+B)^\dag \cdot T_{\alpha,\beta} \cdot (A+B)^\dag\cdot \frac{\beta A-\alpha B}{\alpha+\beta} + \frac{\alpha\beta}{\alpha+\beta}Z \right\Vert\\
&\le&\frac{\Vert T_{\alpha,\beta}\Vert\cdot \Vert (A+B)^\dag (\beta A-\alpha B)\Vert^2}{(\alpha+\beta)^2}+\frac{\alpha\beta}{\alpha+\beta}\Vert Z\Vert\\
&\le& \frac{1}{\alpha+\beta}\left[\frac{\Vert (A+B)^\dag (\beta A-\alpha B)\Vert^2\cdot \Vert A+B\Vert}{\Vert A+B\Vert+(\alpha+\beta)\Vert Z\Vert}+\alpha\beta\right]\Vert Z\Vert\\
&\le& \frac{1}{\alpha+\beta}\Big[\Vert (A+B)^\dag (\beta A-\alpha B)\Vert^2+\alpha\beta\Big]\Vert Z\Vert.\qedhere
\end{eqnarray*}
\end{proof}

\begin{rem}{\rm Let $X,Y\in\mathbb{C}^{n\times n}$ be any Hermitian positive semi-definite matrices and let $\alpha,\beta$ be any positive numbers.
As the numbers of the resistors in electronic circuits can be viewed as positive scalar matrices, it is meaningful to find out a Hermitian positive semi-definite matrix $Z$  such that $\alpha Z\ge X$ and $\beta Z\ge Y$.

One solution to the problem above is $Z_{\alpha,\beta}$, which can be derived directly by \eqref{equ:Introduced Z} and \eqref{equ:Introduced W} as
\begin{equation}\label{equ:Introduced Z-alpha and beta}Z_{\alpha,\beta}=\frac{X}{\alpha}\vee \frac{Y}{\beta}=\frac{\beta X+\alpha Y}{2\alpha\beta}+\left(\frac{\beta X+\alpha Y}{\alpha\beta}\right)^\frac12\, W^\frac12_{\alpha,\beta}\, \left(\frac{\beta X+\alpha Y}{\alpha\beta}\right)^\frac12,
\end{equation}
where
\begin{equation*}\label{equ:Introduced W-alpha and beta}W_{\alpha,\beta}=\left[\left(\frac{\beta X+\alpha Y}{\alpha\beta}\right)^\dag\right]^\frac12\left[\frac{\beta X+\alpha Y}{4\alpha\beta}-X(\beta X+\alpha Y)^\dag Y\right]\left[\left(\frac{\beta X+\alpha Y}{\alpha\beta}\right)^\dag\right]^\frac12.\end{equation*}
}\end{rem}

Now, we consider the general case of the two-sided perturbation of the parallel sum as follows:
\begin{theorem}\label{thm:norm upper bounds with parameters} Suppose that $A,B,X,Y\in \mathbb{C}^{n\times n}$ are  all Hermitian positive semi-definite. Let $E$ be defined by \eqref{eqn:defn of E}. Then
\begin{eqnarray}\label{eqn:norm upper bound os E with parameters-1}\hspace{-2em}\Vert E\Vert &\le& \inf_{\alpha>0,\beta>0}\left[\frac{\Vert (A+B)^\dag (\beta A-\alpha B)\Vert^2\cdot \Vert A+B\Vert}{(\alpha+\beta)\Vert A+B\Vert+(\alpha+\beta)^2\left\Vert Z_{\alpha,\beta}\right\Vert}+\frac{\alpha\beta}{\alpha+\beta}\right]\big\Vert Z_{\alpha,\beta}\big\Vert\\
\label{eqn:norm upper bound os E with parameters-1.5}\hspace{-2em}&\le&\inf_{\alpha>0,\beta>0}\frac{1}{\alpha+\beta}\Big[\Vert (A+B)^\dag (\beta A-\alpha B)\Vert^2+\alpha\beta\Big]\big\Vert Z_{\alpha,\beta}\big\Vert\\
\label{eqn:norm upper bound os E with parameters-2}\hspace{-2em}&=&\inf_{t>0}f(t),
\end{eqnarray}
where $Z_{\alpha,\beta}$ is given by \eqref{equ:Introduced Z-alpha and beta} such that
$Z_{\alpha,\beta}=\frac{1}{\alpha}\left(X\vee \frac{Y}{t}\right)$ for $t=\frac{\beta}{\alpha}$,
and
\begin{equation}\label{equ:defn of f}f(t)=\frac{1}{1+t}\left[\Vert (A+B)^\dag (tA-B)\Vert^2+t\right]\cdot \left\Vert X\vee \frac{Y}{t}\right\Vert.\end{equation}
\end{theorem}
\begin{proof} Let $Z_{\alpha,\beta}$ be given by \eqref{equ:Introduced Z-alpha and beta} for any $\alpha>0$ and $\beta>0$. Then  $\alpha Z_{\alpha,\beta}\ge X$ and $\beta Z_{\alpha,\beta}\ge Y$, which means by \eqref{equ:key equality concerning perturbation of parallel sum}
 than $E\le F_{\alpha,\beta}$ and thus $\Vert E\Vert\le \Vert F_{\alpha,\beta}\Vert$, where $F_{\alpha,\beta}$ is defined by \eqref{eqn:defn of F two dides X} with $Z$ therein be replaced by $Z_{\alpha,\beta}$.
 The desired norm upper bounds follows immediately from Theorem~\ref{thm:two-sided perturbation}.
\end{proof}

Putting $\alpha=\beta=1$ in  \eqref{eqn:norm upper bound os E with parameters-1.5}, we get a corollary as follows:

\begin{corollary}\label{cor:norm upper bound wrt special Z} Suppose that $A,B,X,Y\in \mathbb{C}^{n\times n}$ are  all Hermitian positive semi-definite. Let $E$ be defined by \eqref{eqn:defn of E}. Then
\begin{eqnarray}\label{eqn:2nd sharp of norm E with Z---}\Vert E\Vert\le \mu_{A,B}\Vert X\vee Y\Vert,
\end{eqnarray}
where $X\vee Y$ is  given by  \eqref{equ:Introduced Z}    and  $\mu_{A,B}$  is defined by
\begin{equation}\label{equ:defn of mu A B}\mu_{A,B}=\frac{1}{2}\left[\Vert (A+B)^\dag (A-B)\Vert^2+1\right].\end{equation}
\end{corollary}

\begin{rem}\label{rem:comparison of upper bound wrt Z}{\rm Suppose that $A,B,X,Y\in\mathbb{C}^{n\times n}$ are all Hermitian positive semi-definite. Let $\lambda_{A,B}$ and $\mu_{A,B}$ be defined by  \eqref{equ:defn of lambda A B}   and  \eqref{equ:defn of mu A B}, respectively. Then
\begin{eqnarray*}\mu_{A,B}&=&\frac12\Vert (A+B)^\dag A-(A+B)^\dag B\Vert^2+\frac12\\
&\le&\frac12\left[\Vert (A+B)^\dag A\Vert+\Vert (A+B)^\dag B\Vert\right]^2+\frac12\\
&\le&\Vert (A+B)^\dag A\Vert^2+\Vert (A+B)^\dag B\Vert^2+\frac12\\
&\le&\lambda_{A,B}.
\end{eqnarray*}
The inequalities of $\mu_{A,B}\le \lambda_{A,B}$ and $\Vert X\vee Y\Vert \le \Vert X+Y\Vert$ indicate that upper bound \eqref{eqn:2nd sharp of norm E with Z---} is sharper than the original norm upper bound \eqref{equ:Anderson-Duffin's upper bound}.
}\end{rem}

\begin{rem}{\rm One special case of the two-sided perturbation \eqref{eqn:defn of F two dides X} is $A=B$. Note that $$(A+X):(A+X)-A:A=\frac12 (A+X)-\frac12 A=\frac12 X,$$  so in this case norm upper bound \eqref{eqn:2nd sharp of norm E with Z---} is accurate.
}\end{rem}

\begin{rem}{\rm Given any natural number $n$ and any $k_i>0$ for $i=1,2,3,4$, let $A=k_1 I_n, B=k_2 I_n, X=k_3 I_n$ and $Y=k_4 I_n$,
where $I_n$ is the identity matrix in
$\mathbb{C}^{n\times n}$. Let $E$ be defined by \eqref{eqn:defn of E}.  Then \eqref{eqn:norm upper bound os E with parameters-1} becomes an equation if we put $\alpha=k_3$ and $\beta=k_4$ therein.
}\end{rem}

\section{Numerical examples}\label{sec:Numerical examples}
In this section, we provide two numerical examples as follows.
\begin{ex}\label{ex:parameter to be selected}{\rm For any $t\in (0,\frac{\pi}{2})$, let $A(t),B(t)\in\mathbb{C}^{2\times 2}$ be defined by
$$A(t)=\left(
         \begin{array}{cc}
           \cos(t) & \frac14\sin(t) \\
           \frac14\sin(t) & \cos(t)\\
         \end{array}
       \right)\ \mbox{and}\ B(t)=\left(
         \begin{array}{cc}
           \cos(t) & -\frac14\sin(t) \\
           -\frac14\sin(t) & \cos(t)\\
         \end{array}
       \right).$$
Put $A=A(\frac{\pi}{6}), X=A(\frac{5\pi}{32})-A(\frac{\pi}{6}), B=B(\frac{\pi}{6})$ and $Y=B(\frac{3\pi}{32})-B(\frac{\pi}{6})$. Then $A,B,X$ and $Y$ are all positive definite. Let $E$ and $f$ be defined by \eqref{eqn:defn of E} and \eqref{equ:defn of f}, respectively.  Then
\begin{equation*}E=A\left(\frac{5\pi}{32}\right):B\left(\frac{3\pi}{32}\right)-A\left(\frac{\pi}{6}\right):B\left(\frac{\pi}{6}\right)\ \mbox{and}\ \Vert E\Vert=0.0453,\end{equation*}
and from the graph of $f$ drawn by using Matlab command ``fplot" or by using Matlab command ``fmincon" alternatively, we know that
$f$ gets its infimum around the point $t=6.2197$; that is,  $\inf\{t>0|f(t)\}\thickapprox f(6.2197)=0.0511$.
Thus, a comparison of the errors can be provided as in Table~\ref{tab:1-st comparison}, which shows that for this example, norm upper bound \eqref{eqn:norm upper bound os E with parameters-2} is much better than the other two.
\begin{table}[htbp]
  \caption{Comparison of the errors associated to norm upper bounds \eqref{equ:Anderson-Duffin's upper bound}, \eqref{eqn:norm upper bound os E with parameters-2} and \eqref{eqn:2nd sharp of norm E with Z---}
\label{tab:1-st comparison}}
  \label{tab:foo}
  \centering
  \begin{tabular}{|c|c|c|c|} \hline
   &Upper bound & Upper bound & Upper bound\\  & \eqref{equ:Anderson-Duffin's upper bound} & \eqref{eqn:norm upper bound os E with parameters-2} &\eqref{eqn:2nd sharp of norm E with Z---}  \\ \hline
    Numerical value&  0.2752& 0.0511  & 0.0732 \\\hline
    Relative error&  507.5\% & 12.8\% & 61.6\% \\ \hline
  \end{tabular}
\end{table}
}\end{ex}

\begin{ex}{\rm For any $t\in (-\infty, +\infty)$, let $P(t)\in\mathbb{C}^{2\times 2}$ be the orthogonal projection defined by
\begin{equation*}P(t)=\left(
                        \begin{array}{cc}
                          \cos^2(t) & -\sin(t)\cos(t) \\
                          -\sin(t)\cos(t) & \sin^2(t) \\
                        \end{array}
                      \right).
\end{equation*}
Put $$A=P\left(\frac{\pi}{8}\right), B=P\left(\frac{\pi}{6}\right), X=P\left(\frac{\pi}{4}\right), Y=P\left(\frac{\pi}{3}\right),$$  and let
$E, f$ be defined by \eqref{eqn:defn of E} and \eqref{equ:defn of f}, respectively.
Then $\Vert E\Vert=0.4650$ and  $\inf\{t>0|f(t)\}=f(1)=0.5000$. A comparison of the errors is also provided in Table~\ref{tab:2-nd comparison}, which shows that for this example, norm upper bounds  \eqref{eqn:norm upper bound os E with parameters-2} and \eqref{eqn:2nd sharp of norm E with Z---}  are the same.
\begin{table}[htbp]
  \caption{Comparison of the errors associated to norm upper bounds \eqref{equ:Anderson-Duffin's upper bound}, \eqref{eqn:norm upper bound os E with parameters-2} and \eqref{eqn:2nd sharp of norm E with Z---}\label{tab:2-nd comparison}}
  \label{tab:foo--}
  \centering
  \begin{tabular}{|c|c|c|c|} \hline
   &Upper bound & Upper bound & Upper bound\\  & \eqref{equ:Anderson-Duffin's upper bound} & \eqref{eqn:norm upper bound os E with parameters-2} &\eqref{eqn:2nd sharp of norm E with Z---}  \\ \hline
    Numerical value&  3 & 0.5 & 0.5 \\\hline
    Relative error&  545.2\% & 7.5\% & 7.5\% \\ \hline
  \end{tabular}
\end{table}
}\end{ex}
\section{Concluding remarks}As shown in Theorem~\ref{thm:Almost least common upper bound} that for any two Hermitian positive semi-definite matrices $X,Y\in\mathbb{C}^{n\times n}$, a common upper bound of $X$ and $Y$ can be constructed based on certain $C^*$-algebraic technique. This common upper bound is proved to be strictly less than
$X+Y-X:Y$ whenever  $X:Y$ is non-zero. Furthermore, if $X$ and $Y$ are two positive operators acting on a general Hilbert $C^*$-module \cite[Section~2]{Xu-Wei-Gu}, then Theorem~\ref{thm:Almost least common upper bound} still works in the case that $X+Y$ is Moore-Penrose invertible; or equivalently, $X+Y$ has a closed range \cite[Theorem~2.2]{Xu-Sheng}.

As mentioned early, the positivity of the matrix $H$ defined by \eqref{equ:defn of H} can be derived directly from the factorization formula \eqref{equ:key equality concerning perturbation of parallel sum} for $H$. This newly obtained factorization formula can also be extended to the infinite-dimensional case. More precisely,
 if $A,B,X$ and $Y$ are positive operators such that $A+B+X+Y, A+B$ and $X+Y$ are all Moore-Penrose invertible, then the factorization formula \eqref{equ:key equality concerning perturbation of parallel sum} for $H$ is also valid.

As illustrated by Remark~\ref{rem:comparison of upper bound wrt Z} and two numerical examples in Section~\ref{sec:Numerical examples},
the newly obtained upper bounds \eqref{eqn:norm upper bound os E with parameters-2} and \eqref{eqn:2nd sharp of norm E with Z---} are sharper
than the original one established in \cite[Theorem~31]{Anderson-Duffin}. It is not hard to prove that
norm upper bound \eqref{equ:sharp upper bound for parallel sum wrt positive operators} is also true for positive operators $A$ and $B$ if $A+B$ is Moore-Penrose invertible. Thus in the general setting of Hilbert $C^*$-modules, a generalized version of Theorem~\ref{thm:norm upper bounds with parameters} can also be obtained  provided that the associated operators are all Moore-Penrose invertible.

\vspace{2ex}
\noindent\textbf{Acknowledgments}

\noindent The authors thank the referee for helpful suggestions.

\end{document}